\documentclass[letterpaper, 10 pt, conference]{ieeeconf}
\IEEEoverridecommandlockouts 
\usepackage{graphicx} 
\usepackage{cite}
\usepackage{wrapfig,comment}
\usepackage{amsmath,amssymb,amsfonts,mathtools,textgreek}

\usepackage{amsthm}
\usepackage{algorithmic,enumerate}
\usepackage[dvipsnames]{xcolor}
\usepackage{textcomp}
\usepackage{subcaption}
\usepackage{tikz}
\usetikzlibrary{plotmarks,positioning}
\usepackage{xparse}
\usepackage{hyperref}
\usepackage[normalem]{ulem}

\newtheorem{lemma}{Lemma}
\newtheorem{theorem}{Theorem}

\newtheorem{assumption}{Assumption}
\newtheorem{corollary}{Corollary}

\newtheorem{problem}{Problem}
\newtheorem{fact}{Fact}

\title{\LARGE \bf Risk-Aware LQR with Temporally Coupled States*
}

\author{Chuanning Wei$^1$,
Kin Fung Li$^1$,
Dionysis Kalogerias$^2$,
Margaret P. Chapman$^1$
\thanks{$^*$We acknowledge the support of the Natural Sciences and Engineering Research Council of Canada (NSERC), [funding reference number RGPIN-2022-04140]. Cette recherche a été financée par le Conseil de recherches en sciences naturelles et en génie du Canada (CRSNG), [numéro de référence RGPIN-2022-04140]. We acknowledge the support of the Edward S. Rogers Sr. Department of Electrical and Computer Engineering, University of Toronto.
}
\thanks{$^1$C.W., K.F.L., and M.P.C. are with the Edward S. Rogers Sr. Department of Electrical and Computer Engineering, University of Toronto, 10 King's College Road, Toronto, ON, M5S 3G4, Canada. Contact emails: chuanning.wei[at]mail.utoronto.ca (C. Wei), calvinkf.li[at]mail.utoronto.ca (K.F. Li), 
margaret.chapman[at]utoronto.ca (M.P. Chapman).}
\thanks{$^2$D.K. is with the Department of Electrical and Computer Engineering, Yale University, New Haven, CT, 06520, USA. Contact email: dionysis.kalogerias[at]yale.edu.}
\thanks{Corresponding author: M.P. Chapman.}
}

\begin{document}
\maketitle
\thispagestyle{empty}
\pagestyle{empty}

\begin{abstract}                        
We formulate and solve a discrete-time linear-quadratic regulation (LQR) problem in a finite horizon that penalizes temporal variability and stochastic variability of the state trajectory. Our approach enables one to strike a balance between regulating the state and reducing temporal variability,
with explicit sensitivity to risk. We achieve this by extending a risk measure called predictive variance to a setting with temporally coupled states. Numerical examples demonstrate the effect of temporal coupling in both
risk-aware and risk-neutral control settings. 
\end{abstract}

\section{INTRODUCTION}\label{sec:intro}
In state regulation of a control system, the behavior of the state trajectory can differ as a result of design choices. In particular, one may prefer trajectories that vary slowly over time, if abrupt changes in the state are undesirable. Related applications include robotics, e.g., designing shorter paths \cite[Eq. 8]{dixit2024step}, and building climate control, e.g., preferring gradual temperature changes \cite{somasundaram2023raising}. The design of slowly varying, continuous-time trajectories (small $|\dot{x}|$) has been studied in \cite{guarino1996nonlinear, guarino2014discrete, consolini2024time}, among many others. In this paper, we study a discrete-time, stochastic control problem that rewards slowly varying trajectories. The setting of interest concerns standard linear systems, but we evaluate their performance in a manner that deviates from a classical stochastic viewpoint.

Classical stochastic optimal control summarizes uncertain performance on average and thus is not designed to manage rare events  
\cite{van2015distributionally, majumdar2020should, wang2021riskaverse, smith2023exponential,akella2025risk}. For example, the linear-quadratic-Gaussian problem 
minimizes the expectation of the LQR cost $\mathcal{J} \coloneqq x_N^\top Q x_N + \sum_{t=0}^{N-1} x_t^\top Q x_t + u_t^\top R u_t$ and 
so lacks explicit sensitivity
to the stochastic variability of $\mathcal{J}$. This limitation motivates the use of more nuanced measures of \emph{risk} (informally, tools to quantify probability distributions) in controller design  \cite{shapiro2021lectures, royset2025Risk, wang2021riskaverse,chapman2024risk,tsiamis2025linear,akella2025risk, biswas2023ergodic}.
This paper uses a variance-like measure to quantify risk for a fully observable, linear system with additive noise.

Reducing stochastic variability in system performance can be dated back to the
minimal cost variance (MCV) 
problem studied by Sain in 1966, which minimizes the variance of an LQR cost subject to an expected cost constraint \cite{sain1966control}.
Another classical example is the linear-exponential-quadratic-Gaussian (LEQG) problem \cite{whittle1981risk}. This problem minimizes the objective 
$-2\theta^{-1} \log\mathbb{E}(\exp(-\theta\mathcal{J}/2))$ with a risk-aversion parameter $\theta < 0$,
which approximates a mean-variance objective locally \cite{whittle1981risk}. LEQG is closely related to other problems in control theory. There is a correspondence between LEQG and a mixed $\mathcal{H}_2$/$\mathcal{H}_\infty$ control problem \cite{glover1988state, zhang2021policy}. Also, both the MCV problem and the LEQG problem (under a series expansion of the LEQG objective) can be viewed as special cases of controlling the statistics of the distribution of an LQR cost \cite{sain1995cumulants, kang2013statistical}. However, the Gaussian noise assumptions in \cite{whittle1981risk, kang2013statistical} imply that these controllers are not designed to handle more general noise distributions that may be skewed or heavy-tailed.

A modern approach for managing stochastic variability in linear system performance is to incorporate \emph{predictive variance} into an optimal control problem \cite{tsiamis2020risk}. This measure assesses the variability of the state energy $x_t^\top Q x_t$ relative to its predicted mean, taking the form of $\mathbb{E}((x_t^\top Q x_t - \mathbb{E}(x_t^\top Q x_t|h_{t-1}))^2)$ \cite{tsiamis2020risk}. Predictive variance is an attractive risk measure for fully observable, linear systems with additive noise. First, the noise need not be Gaussian \cite{tsiamis2020risk}.
Second,  
an optimal controller that is affine state-feedback and depends on the skewness of the noise can be computed in closed form  \cite{tsiamis2020risk}. 
Third, predictive variance assesses risk in terms of stochastic variability in a dynamic setting \cite{tsiamis2020risk}. The predictive variance approach has been extended to the infinite-horizon case \cite{zhao2021infinite}, partially observable systems \cite{tsiamis2025linear}, and interconnected systems \cite{patel2024risk, wang2025risk}.

In the current work, we connect predictive variance \cite{tsiamis2020risk} to the aim of rewarding state trajectories that vary gradually over time. To formulate this problem mathematically, we begin by defining a state energy and its predictive variance for a time period of stages rather than the current stage alone. The resulting cost functional now penalizes a \emph{period-wise} behavior, and the relevant state variables are said to be \emph{temporally coupled}. For historical context, such a behavior is discussed in \cite{bacchus1996rewarding}. To illustrate, the maximum power drawn from an electric grid \cite{jones2020extensions} over a period corresponds to a period-wise behavior, whereas the instantaneous power corresponds to a stage-wise behavior. To penalize such history-dependent behaviors, we will consequently treat a non-Markovian optimal control problem, which does not immediately admit a standard dynamic programming (DP) solution \cite{bacchus1996rewarding}.

State augmentation has been used to treat history-dependent objectives, e.g., see \cite{bacchus1996rewarding,jones2020extensions, bauerle2011markov, bauerlerieder, chapman2023on}, and its applications to particular problem settings can be nontrivial. In our case, the extended predictive variance, which assesses period-wise stochastic variability, is a nonstandard history-dependent cost function; i.e., it is \emph{not} defined in the form of $c_t(\xi_t, u_t)$ for some enlarged state $\xi_t$. This extra layer of complexity requires a reformulation, involving a measure-theoretic derivation (Appendix), so that the risk penalty can enjoy a standard, augmented-state form (to be presented in Problem \ref{prob:mean-pv-quad}). Next, this reformulated objective functional is placed in a standard DP problem with augmented states, which we show is the same as Problem \ref{prob:mean-pv-quad} (Theorem \ref{prop:equiv}) through a recovery of the original dynamics. We also allow one to change the number of augmented states, governed by an integer $\mathbf{k}$, whose effects can be empirically assessed (Fig. \ref{fig:comparison}).

\textit{Contribution:} We formulate and solve a nonstandard stochastic LQR problem in a risk-aware manner, i.e., with explicit sensitivity to rarer outcomes in addition to what is expected. The problem deviates from a standard regulation problem in that it allows one to encode a preference for slowly varying state trajectories along with state regulation, e.g., by penalizing both $(x_t - x_{t-1})^\top Q (x_t - x_{t-1})$ and $x_t^\top Q x_t$. We generalize this penalty using a convex quadratic cost that contains historical states over a period, a cost called the \emph{sequential state energy}; recall those states are said to be temporally coupled.
We quantify risk by extending predictive variance \cite{tsiamis2020risk,tsiamis2025linear},
a measure of stochastic variability, to be defined for the sequential state energy. The extended predictive variance is  
a period-wise risk concept instead of a stage-wise one. Thus, our problem differs from risk-aware control problems that concern stage-wise risk concepts, such as those in \cite{ruszczynski2010risk, fujimoto2011optimal, tsiamis2020risk, samuelson2018safety, van2015distributionally, bauerle2022markov, jia2024decentralized,tsiamis2025linear}. Since the (extended) predictive variance is defined via conditional expectation,
we derive an analytical expression for it, which is convex quadratic in the temporally coupled states (Lemmas \ref{lemma:delta-t-decomposition}--\ref{lemma:pred-var-quad}).
This discovery allows us to reformulate our nonstandard problem (Problem~\ref{prob:mean-pv}) into a form with a well-known, analytical solution (Theorem~\ref{prop:equiv}, Corollary~\ref{thm:main}).

\textit{Organization:}
Section~\ref{sec:model} formulates the problem of interest. 
Sections~\ref{sec:reduction} and~\ref{sec:opt-control} develop the solution. Section \ref{sec:num} presents numerical examples.

\textit{Notation:}
$\mathbb{N} \coloneqq \{1,2,3,\dots\}$.
$\mathbb{R}^{n \times m}$ is the set of $n \times m$ real matrices.
$0_{n \times m}$ is the $n \times m$ matrix of all zeros. $I_n$ is the $n \times n$ identity matrix. $\mathcal{S}_n$ ($\mathcal{S}_n^+$) is the set of $n \times n$ real symmetric positive semidefinite (definite) matrices.
$(\cdot)^\top$ denotes the matrix transpose. $\oplus$ denotes the matrix direct sum \cite[Def.~0.9.2]{horn2012matrix}.
$|\cdot|$ denotes the Euclidean norm. For column vectors $v_1,\dots,v_m$, we abbreviate $[v_1^\top, \dots, v_m^\top]^\top$ as $(v_1,\dots,v_m)$.
$\mathbb{E}(\cdot)$ ($\mathrm{var}(\cdot)$) denotes expectation (variance). $y \in \mathcal{L}^2(\mathcal{F})$ specifies that $y$ is an $\mathcal{F}$-measurable random vector  
such that $\mathbb{E}(|y|^2) < \infty$. $\sigma(\cdot)$ denotes the $\sigma$-algebra generated by a random vector \cite[Def. 6.4.1]{ash1972probability}. 
a.s. stands for almost surely. 
We fix a parameter $\mathbf{k}$ throughout Sections~\ref{sec:model}--\ref{sec:opt-control} and often suppress it from our notation for brevity. 
$\lambda \in [0,\infty)$ and $R \in \mathcal{S}_m^+$ are parameters representing the importance of penalizing risk and control effort, respectively.
The symbol $\mathbf{u}$ is shorthand for $(u_0,\dots,u_{N-1})$.

\section{PROBLEM FORMULATION} \label{sec:model}

Consider a fully observable, stochastic linear system 
\begin{equation} \label{linsys}
    x_{t+1} = Ax_t + Bu_t + w_t, \quad t \in \{0,\dots,N-1\},
\end{equation}
where $N \in \mathbb{N}$ is the horizon length, $A \in \mathbb{R}^{n \times n}$ and $B \in \mathbb{R}^{n \times m}$ are deterministic, $x_t$ is a $\mathbb{R}^n$-valued random vector (state), $u_t$ is a $\mathbb{R}^m$-valued random vector (control), and $w_t$ is a $\mathbb{R}^n$-valued random vector (disturbance). 
All random vectors are defined on a probability space $(\Omega,\mathcal{G},\mathbb{P})$, and 
$x_0,w_0,\dots,w_{N-1}$ are independent \cite[p. 18]{kumar2015stochastic}. 
We assume that $w_0,\dots,w_{N-1}$ are identically distributed without imposing a particular distribution, although we impose an integrability condition below.
We define $\mathcal{F}_t \coloneqq \sigma(h_t)$, where $h_t \coloneqq (x_0,u_0,\dots,x_{t-1}, u_{t-1},x_t)$ if $t \in \{1,\dots,N\}$ and $h_0 \coloneqq x_0$.  
We assume that the initial state is deterministic, i.e., $x_0 = \mathbf{x}_0$ a.s., where $\mathbf{x}_0 \in \mathbb{R}^n$ is arbitrary.  
We also impose the following conditions:
\begin{assumption}\label{assu1}
    For every $t \in \{0, \ldots, N-1\}$, \emph{(A1)} $u_t$ is causal and square-integrable, i.e., $u_t \in \mathcal{L}^2(\mathcal{F}_t)$; and \emph{(A2)} $w_t$ satisfies $\mathbb{E}(|w_t|^4) < \infty$.
\end{assumption}

While the square-integrability of $x_0$ and all controls and disturbances ensures that each $x_t$ is square-integrable due to  
\eqref{linsys}, 
the stronger assumption (A2) enables the use of a risk measure called predictive variance \cite{tsiamis2020risk}. 

\subsection{Sequential State Energy and its Predictive Variance}\label{sec2a}
As motivated in Section~\ref{sec:intro}, we study an optimal control problem (to be defined in Section~\ref{seclabelproblem}) that penalizes temporal variability and stochastic variability of the state trajectory. 
To quantify temporal variability,
we extend the state energy $x_t^\top Q x_t$ as follows: 
Given a \emph{coupling length} $\mathbf{k} \in \{0,\ldots,N\}$, we define for each $t \in \{0,\ldots,N\}$ the \emph{sequential state energy}
\begin{equation} \label{eq:zt}
    z_t \coloneqq \eta_t^\top \mathcal{Q}_t \eta_t,
\end{equation}
where 
\begin{equation}\label{eta_def}
    \eta_t \coloneqq ( x_t, \dots,  x_{t-\mathbf{k}_t})
\end{equation} 
is the \emph{truncated state history} containing $n_t \coloneqq n(\mathbf{k}_t + 1)$ elements with 
$\mathbf{k}_t \coloneqq \min\{\mathbf{k},t\}$. $\mathcal{Q}_t$ is the first $n_t$ rows and first $n_t$ columns of the matrix $\mathcal{Q} \in \mathcal{S}_{n(\mathbf{k}+1)}$ denoted by 
\begin{equation}\label{eq:scriptQ}
    \mathcal{Q} \coloneqq \begin{bmatrix}
        Q_{00} & \cdots & Q_{0 \mathbf{k}} \\
        \vdots & \ddots & \vdots \\
        Q_{\mathbf{k} 0} & \cdots & Q_{\mathbf{k} \mathbf{k}} \\
    \end{bmatrix}. 
\end{equation}
Since $\mathcal{Q} \in \mathcal{S}_{n(\mathbf{k}+1)}$ by assumption, it follows that $\mathcal{Q}_t \in \mathcal{S}_{n_t}$ and  
$Q_{ii} \in \mathcal{S}_n$ for each $i$ \cite[Observation 7.1.2]{horn2012matrix}. 
$\eta_t$ encodes the state trajectory in at most $\mathbf{k}+1$ recent times; the states in $\eta_t$ are said to be temporally coupled (introduced previously).
For example, given $Q, \bar{Q} \in \mathcal{S}_n$, we can express the one-step temporally coupled cost
\begin{equation} \label{eq:first-order-ex}
    x_t^\top Q x_t + (x_t - x_{t-1})^\top \bar{Q} (x_t - x_{t-1})
\end{equation}
in the form of \eqref{eq:zt} using $\eta_t = (x_t, x_{t-1})$ and $\mathcal{Q}_t = Q \oplus 0_{n \times n} + [-I_n \; I_n]^\top \bar{Q} [-I_n \; I_n]$.

Next, we quantify risk in terms of the stochastic variability of the recent state trajectory. 
We define the \emph{predictive variance of the sequential state energy} by $\mathbb{E}(\Delta_t^2)$, where $\Delta_t \coloneqq z_t - \mathbb{E}(z_t | \mathcal{F}_{t-1})$, for every $t \in \{1,\dots,N\}$. If $\mathbf{k} = 0$ and $Q_{00}=Q$, then the sequential state energy $z_t$~\eqref{eq:zt} reduces to the state energy $x_t^\top Q x_t$, so that $\Delta_t = x_t^\top Q x_t - \mathbb{E}(x_t^\top Q x_t | \mathcal{F}_{t-1})$ and our definition of predictive variance coincides with the one defined for state energy in \cite[Eq.~3]{tsiamis2020risk}.
Our construction enables one to strike a balance between regulating the state and reducing temporal variability  
by permitting a cost function like \eqref{eq:first-order-ex}. 

\subsection{Risk-Aware Optimal Control Problem} \label{seclabelproblem}

Using the predictive variance $\mathbb{E}(\Delta_t^2)$ defined in Section \ref{sec2a}, we present the following optimal control problem: 

\begin{problem} \label{prob:mean-pv}
    Under Assumption (A2),  
    consider 
    \begin{align*}
         & \; \inf_\mathbf{u} && \mathcal{J}_\lambda(\mathbf{u}) \\
        & \mathrm{\;\;subject\;to} && \mathrm{dynamics}~\eqref{linsys}, \; x_0 = \mathbf{x}_0 \mathrm{ \;a.s.}, \\
        & && u_t \in \mathcal{L}^2(\mathcal{F}_t), \; t \in \{0,\ldots,N-1\} ,
    \end{align*}
    where the objective $\mathcal{J}_\lambda(\mathbf{u})$ is defined by
    \begin{equation} \label{eq:obj}
        \textstyle     \mathcal{J}_\lambda(\mathbf{u}) \coloneqq \mathbb{E}\left(z_N + \sum_{t=0}^{N-1}z_t + u_t^\top R u_t\right) + \lambda \mathbb{E}\left(\sum_{t=1}^{N} \Delta_{t}^2\right) .
    \end{equation}
\end{problem}

Problem~\ref{prob:mean-pv} allows for trading off various preferences related to $\lambda$, $\mathcal{Q}_t$, and $\mathbf{k}$. 
$\lambda$ represents the relative importance of penalizing $z_t$~\eqref{eq:zt} and control effort on average versus the stochastic variability of $z_t$. Hence, Problem~\ref{prob:mean-pv} resembles a \emph{mean-risk} optimization problem, which gives flexibility to make a compromise between reducing a cost on average versus reducing its uncertainty~\cite[Sec.~6.2]{shapiro2021lectures}.
$\mathcal{Q}_t$ and $\mathbf{k}$
offer finer control over the state trajectory, in addition to the typical goal of state regulation.  
$z_t$~\eqref{eq:zt} expands to
\begin{equation} \label{eq:zt-expanded}
    z_t = \sum_{i=1}^{\mathbf{k}_t} \sum_{j=1}^{\mathbf{k}_t} x_{t-i}^\top Q_{ij} x_{t-j} + 2 \sum_{i=1}^{\mathbf{k}_t} x_{t-i}^\top Q_{i0} x_t + x_t^\top Q_{00} x_t,
\end{equation}
where each term is bilinear in two states at possibly different times.
$x_t$ is coupled with a previous state $x_{t-i}$ via $x_{t-i}^\top Q_{i0} x_t$, but only $x_t^\top Q_{00} x_t$ would take effect at time $t$ in standard state regulation. 
In the special case of \eqref{eq:first-order-ex}, as $\bar{Q}$ dominates $Q$, reducing the one-step state difference is preferred over driving the state to the origin. 

Problem~\ref{prob:mean-pv} explicitly quantifies  
risk  
in terms of stochastic variability while accounting for temporal variability, differing from many approaches.
For example, approaches based on CVaR are designed to be sensitive 
to the (one-sided) tail but not the rest of a probability distribution
\cite{van2015distributionally, samuelson2018safety, chapman2022tac, kishida2022risk, chapman2023on, akella2025risk}.
In contrast, variance-like measures, particularly our (extended) predictive variance,
consider both sides of a distribution and  
quantify
stochastic (two-sided) variability.
This also differs from a nested risk representation \cite{ruszczynski2010risk, bauerle2022markov}, which takes the form of $\rho_0(y_0 + \rho_1(y_1 + \dots ) )$ with $y_t$ a random variable and $\rho_t$ a map assessing risk at time $t$, where the nested structure can be difficult to interpret \cite{iancu2015tight, wang2021riskaverse}.
It is pertinent to assess the risk of a control system in an interpretable way, especially when the system represents a real-world process \cite{chapman2022tac, majumdar2020should}. Interpretable, trajectory-wise risk assessments were developed using CVaR \cite{bauerle2011markov, chapman2022tac, chapman2023on} and expected utility measures \cite{bauerlerieder} for Markov decision processes 
but do not admit analytical solutions in general.

\section{REFORMULATION TO A QUADRATIC FORM} \label{sec:reduction}

To reformulate Problem~\ref{prob:mean-pv}, first we show that its objective enjoys a quadratic representation in appropriate variables in three steps. First, we decompose $\Delta_t^2$ into a sum of three integrable terms and evaluate their expectations (Lemma~\ref{lemma:delta-t-decomposition}). Then, we show that the predictive variance $\mathbb{E}(\Delta_t^2)$ is convex and quadratic in $\eta_t$ (Lemma~\ref{lemma:pred-var-quad}). Finally, we reformulate the objective $\mathcal{J}_\lambda(\mathbf{u})$ as a sum of convex quadratic functions in $\eta_t$ and $u_t$ (Lemma~\ref{lemma:obj-quad}). These results are expressed in terms of the following stationary statistics,  
which are finite under Assumption~\ref{assu1}: $\bar{w}  \coloneqq \mathbb{E}(w_{t-1})$, $\Sigma \coloneqq \mathbb{E}(d_t d_t^\top)$, $\gamma  \coloneqq \mathbb{E}(d_t d_t^\top Q_{00} d_t)$, and $\delta \coloneqq \mathrm{var}(d_t^\top Q_{00} d_t)$, where $d_t \coloneqq w_{t-1} - \bar{w}$ for every $t\in\{1,\ldots,N\}$ \cite{tsiamis2020risk}. 
The results are grounded in elementary measure and probability theory \cite{ash1972probability}.

\begin{lemma} \label{lemma:delta-t-decomposition}
    Let Assumption~\ref{assu1} hold and $t \in \{1,\dots,N\}$ be given. Define the prediction error of the state energy  
    by $y_t \coloneqq x_t^\top Q_{00} x_t - \mathbb{E}(x_t^\top Q_{00} x_t | \mathcal{F}_{t-1})$. Also, define
    \begin{subequations} \label{term123} \begin{align}
        \Delta_{t,1} & \coloneqq y_t^2, \label{term1} \\
        \Delta_{t,2} & \coloneqq \textstyle 4 \left(\sum_{i=1}^{\mathbf{k}_t} x_{t-i}^\top Q_{i0} d_t\right)^2  && \text{if $\mathbf{k}_t \geq 1$, and} \label{term2} \\
        \Delta_{t,3} & \coloneqq \textstyle 4 \sum_{i=1}^{\mathbf{k}_t} x_{t-i}^\top Q_{i0} d_t y_t && \text{if $\mathbf{k}_t \geq 1$}. \label{term3}
    \end{align}\end{subequations}
     If $\mathbf{k}_t \ge 1$, then $\Delta_t^2 = \Delta_{t,1} + \Delta_{t,2} + \Delta_{t,3}$ a.s.,
    where $\Delta_{t,1}$, $\Delta_{t,2}$, and $\Delta_{t,3}$ are integrable,
    \begin{subequations}
    \begin{align}
        \mathbb{E}(\Delta_{t,1}) & = 4\mathbb{E}(x_t^\top Q_{00} \Sigma Q_{00} x_t + x_t^\top Q_{00} \gamma) + \vartheta, && \label{term1-expt} \\
        \mathbb{E}(\Delta_{t,2}) & \textstyle = 4 \sum_{i=1}^{\mathbf{k}_t} \sum_{j=1}^{\mathbf{k}_t} \mathbb{E}(x_{t-i}^\top Q_{i0} \Sigma Q_{0j} x_{t-j}), && \label{term2-expt} \\
        \mathbb{E}(\Delta_{t,3}) & \textstyle = 4\sum^{\mathbf{k}_t}_{i=1} \mathbb{E}\big(x_{t-i}^\top Q_{i0} (\gamma + 2\Sigma Q_{00} x_t)\big) , && \label{term3-expt}
    \end{align}
    \end{subequations}
    and $\vartheta \coloneqq \delta - 4 \mathrm{tr}((\Sigma Q_{00})^2)$.
    If $\mathbf{k}_t = 0$, then $\Delta_t^2 = \Delta_{t,1}$ is integrable and $\mathbb{E}(\Delta_t^2)$ is given by \eqref{term1-expt}.
\end{lemma}

The proof of Lemma~\ref{lemma:delta-t-decomposition} is based on measure-theoretic principles and is presented in the Appendix. The special case of $\mathbf{k}_t=0$ was derived in \cite[Appendix]{tsiamis2020risk} and is stated for completeness, while the terms $\Delta_{t,2}$ \eqref{term2} and $\Delta_{t,3}$ \eqref{term3} arise due to the temporally coupled states. Since $x_{t-i}$ is $\mathcal{F}_{t-1}$-measurable for $i \in \{1,\dots,\mathbf{k}_t\}$, the submatrices $Q_{ij}$ for $i$ and $j$ in $\{1,\dots,\mathbf{k}_t\}$ are absent from \eqref{term123}. However, these submatrices still appear in $z_t$ in general. 
The above results can be expressed compactly. First, for each $t\in\{0,\ldots,N\}$, we define $\zeta_t \coloneqq 4[
        Q_{00}^\top \; \cdots \; Q_{\mathbf{k}_t 0}^\top
    ]^\top \gamma$ and
\begin{equation*}
    H_t  \coloneqq 4 \begin{bmatrix}
        Q_{00} \Sigma Q_{00} & \cdots & Q_{00} \Sigma Q_{0 \mathbf{k}_t} \\
        \vdots & \ddots & \vdots \\
        Q_{\mathbf{k}_t 0} \Sigma Q_{00} & \cdots & Q_{\mathbf{k}_t 0} \Sigma Q_{0 \mathbf{k}_t}
    \end{bmatrix}. 
\end{equation*}

\begin{lemma} \label{lemma:pred-var-quad}
Under Assumption~\ref{assu1}, for any $t \in \{1,\ldots,N\}$, we have $H_t \in \mathcal{S}_{n_t}$ and 
$ \mathbb{E}(\Delta_t^2) = \mathbb{E}(\eta_t^\top H_t \eta_t + \zeta_t^\top \eta_t) + \vartheta$.
\end{lemma}

\emph{Proof.}
To confirm that $H_t \in \mathcal{S}_{n_t}$, apply the decomposition $H_t = 4 \bar{Q}_t^\top \Sigma \bar{Q}_t$, where $\bar{Q}_t \coloneqq [
    Q_{00} \; \cdots \; Q_{0\mathbf{k}_t}
]$ and $\Sigma \in \mathcal{S}_n$.
The quadratic form of $\mathbb{E}(\Delta_t^2)$ follows from Lemma~\ref{lemma:delta-t-decomposition} and additivity of expectations \cite[Th. 1.6.3]{ash1972probability}.  
\qed

\begin{lemma} \label{lemma:obj-quad}
    Under Assumption~\ref{assu1}, $\mathcal{J}_\lambda(\mathbf{u})$~\eqref{eq:obj} satisfies
    \begin{equation} \label{eq:obj-alt}
      \textstyle  \mathcal{J}_\lambda(\mathbf{u}) = \mathbb{E}\left(c_N(\eta_N) + \sum_{t=0}^{N-1}c_t(\eta_t,u_t)\right) + c \;,
    \end{equation}
    where $c_N : \mathbb{R}^{n_N} \rightarrow \mathbb{R}$ and $c_t : \mathbb{R}^{n_t} \times \mathbb{R}^m \rightarrow \mathbb{R}$ are quadratic, $\lambda$-parameterized functions defined by
    \begin{equation} \label{eq:stage-cost} \begin{aligned} 
        c_N(\eta) & \coloneqq \eta^\top \mathcal{Q}_{\lambda,N} \eta + \lambda \zeta_N^\top \eta  \quad   \text{and} \\
        c_t(\eta,u) & \coloneqq \eta^\top \mathcal{Q}_{\lambda,t} \eta + \lambda \zeta_t^\top \eta + u^\top R u,
    \end{aligned}\end{equation}
    respectively, $\mathcal{Q}_{\lambda,t} \coloneqq \mathcal{Q}_t + \lambda H_t \in \mathcal{S}_{n_t}$ is the inflated state history penalty matrix for each $t \in \{0,\dots,N\}$, and $c \coloneqq \lambda(N\vartheta - 4\mathbf{x}_0^\top Q_{00} \Sigma Q_{00} \mathbf{x}_0 - 4\gamma^\top Q_{00} \mathbf{x}_0)$ is a real number.
\end{lemma}

\emph{Proof.} Apply Lemma~\ref{lemma:pred-var-quad} to the definition of $\mathcal{J}_\lambda(\mathbf{u})$ \eqref{eq:obj}, particularly using the fact that $\eta_0,\dots,\eta_N$ 
are square-integrable random vectors under Assumption \ref{assu1}. \qed

Substituting \eqref{eq:obj-alt} into Problem~\ref{prob:mean-pv} yields the following problem (where we omit the constant term $c$):

\begin{problem} \label{prob:mean-pv-quad}
    Under Assumption (A2),
    consider the optimal control problem, with $c_t$ defined in \eqref{eq:stage-cost} and $\eta_t$ defined to be $(x_t,\dots,x_{t-\mathbf{k}_t})$ for each $t$:
    \begin{align*}
        \mathcal{J}_\lambda^*(\mathbf{x}_0) \coloneqq  & \inf_\mathbf{u} && \textstyle \mathbb{E}\left(c_N(\eta_N) + \sum_{t=0}^{N-1}c_t(\eta_t,u_t)\right) \\
          & \; \mathrm{subject\;to} && \mathrm{dynamics}~\eqref{linsys}, \;x_0 = \mathbf{x}_0 \mathrm{\; a.s.},  \\
          & && u_t \in \mathcal{L}^2(\mathcal{F}_t), \; t \in \{0,\ldots,N-1\} .
    \end{align*}
\end{problem}

Observe that the objectives of Problems~\ref{prob:mean-pv} and \ref{prob:mean-pv-quad} merely differ by a constant, and the constraints of these problems are identical. Thus, if $\mathbf{u}^*$ is optimal for Problem~\ref{prob:mean-pv}, then $\mathbf{u}^*$ is optimal for Problem~\ref{prob:mean-pv-quad}, and vice versa. The objective of Problem~\ref{prob:mean-pv-quad} enjoys a particularly convenient form, being a sum of convex quadratic costs in $\eta_t$ and $u_t$.

\section{RISK-AWARE CONTROLLER SYNTHESIS} \label{sec:opt-control}

We show that an optimal controller can be computed via a Riccati recursion and only requires the truncated state history $\eta_t$ instead of the full history $h_t$. 
This is achieved by reformulating the system dynamics and the control constraints, using $\eta_t$ as an augmented state.

Despite the dynamics~\eqref{linsys} being Markovian (in the original state $x_t$),
dynamic programming  
cannot be applied directly to Problem~\ref{prob:mean-pv-quad} in its current form because $c_t$ is a function of $\eta_t$ rather than $x_t$ alone.
To circumvent this issue, we express the dynamics~\eqref{linsys} in terms of $\eta_t$, regarding $\eta_t$ as a $\mathbb{R}^{n_t}$-valued random augmented state vector. For every $t \in \{0,\dots,N-1\}$, we define
\begin{align*}
    A_t & \coloneqq \begin{cases}
        \begin{bmatrix}
            A & 0_{n \times n\mathbf{k}_t}
        \end{bmatrix} & \text{if }\mathbf{k}_t \ge 1, \\
        A & \text{if }\mathbf{k}_t = 0,
    \end{cases} \quad \text{for } \mathbf{k} \in \{0,\dots,N\}, \\
    D_t & \coloneqq \begin{cases}
        \begin{bmatrix}
            I_{n\mathbf{k}} & 0_{n\mathbf{k} \times n}
        \end{bmatrix} & \text{if }t \ge \mathbf{k}, \\
        I_{n(t+1)} & \text{if }t < \mathbf{k},
    \end{cases}
    \quad \text{for } \mathbf{k}\in \{1,\dots,N\}.
\end{align*}
Observe that $A_t$ satisfies the relation $A_t \eta_t = A x_t$, and if $t \ge \mathbf{k}$, then $D_t \eta_t$ equals $\eta_t$ without the earliest state $x_{t-\mathbf{k}}$. Then, the augmented dynamics are
\begin{equation} \label{augsys}
    \eta_{t+1} = f_t(\eta_t,u_t,w_t), \quad t \in \{0,\dots,N-1\},
\end{equation}
where we define $f_t: \mathbb{R}^{n_t} \times \mathbb{R}^m \times \mathbb{R}^n \to \mathbb{R}^{n_{t+1}}$ by
  $  f_t(\eta,u,w) \coloneqq \tilde{A}_t \eta + \tilde{B}_t u + \tilde{C}_t w $
and the matrices by
\begin{equation*}
    \tilde{A}_t \coloneqq \begin{bmatrix}
        A_t \\ D_t
    \end{bmatrix}, \;
    \tilde{B}_t \coloneqq \begin{bmatrix}
        B \\ 0_{n\mathbf{k}_{t+1} \times m}
    \end{bmatrix}, \;
    \tilde{C}_t \coloneqq \begin{bmatrix}
        I_n \\ 0_{n\mathbf{k}_{t+1} \times n}
    \end{bmatrix}
\end{equation*}
if $\mathbf{k} \in \{1,\dots,N\}$,
whereas $\tilde{A}_t \coloneqq A$, $\tilde{B}_t \coloneqq B$, and $\tilde{C}_t \coloneqq I_n$ if $\mathbf{k} = 0$.
The augmented dynamics \eqref{augsys} are Markovian (in $\eta_t$) and linear. 
However, unlike a standard LQR set-up, the state space $\mathbb{R}^{n_t}$ of \eqref{augsys} is time-varying.
For convenience, we define $\mathcal{G}_t \coloneqq \sigma(g_t)$ with $g_t \coloneqq (\eta_0,u_0,\dots,\eta_{t-1},u_{t-1},\eta_t)$ if $t \in \{1,\dots,N\}$ and $g_0 \coloneqq \eta_0$, analogous to the definitions of $\mathcal{F}_t$ and $h_t$, respectively.
As a final step in the reformulation, 
we pose the following problem 
and formalize its connection to
Problems~\ref{prob:mean-pv}--\ref{prob:mean-pv-quad}.

\begin{problem} \label{prob:mean-pv-aug}
    Under Assumption (A2), 
    consider the optimal control problem, with $c_t$ defined in \eqref{eq:stage-cost} for each $t$:
    \begin{align*}
          & \inf_\mathbf{u} && \textstyle \mathbb{E}\left(c_N(\eta_N) + \sum_{t=0}^{N-1}c_t(\eta_t,u_t)\right) \\
          & \; \mathrm{subject\;to} && \mathrm{dynamics}~\eqref{augsys}, \; \eta_0 = \mathbf{x}_0 \mathrm{\; a.s.}, \\
          & && u_t \in \mathcal{L}^2(\mathcal{G}_t), \; t \in \{0,\ldots,N-1\} .
    \end{align*}
\end{problem}

\begin{theorem} \label{prop:equiv}
    If $\mathbf{u}^*$ is optimal for Problem~\ref{prob:mean-pv}, \ref{prob:mean-pv-quad}, or \ref{prob:mean-pv-aug}, then $\mathbf{u}^*$ is optimal for the remaining two problems.
\end{theorem}

\emph{Proof.}
It suffices to prove the statement for Problems 2 and 3 when $\mathbf{k} \geq 1$.
Denote the objective and constraint set of Problem $i \in \{2,3\}$ as $\mathcal{J}_i(\mathbf{u})$ and $\mathcal{C}_i$, respectively. First, $\mathcal{C}_2 \subseteq \mathcal{C}_3$ and $\mathcal{J}_2(\mathbf{u}) = \mathcal{J}_3(\mathbf{u})$ for every $\mathbf{u} \in \mathcal{C}_2$ follow from applying $\eta_t \coloneqq ( x_t, \dots, x_{t-\mathbf{k}_t})$ \eqref{eta_def}. Next, we argue that $\mathcal{C}_3 \subseteq \mathcal{C}_2$. We must \emph{not} use \eqref{eta_def} in this part because the $\eta_t$'s \emph{within} Problem~\ref{prob:mean-pv-aug} are intermediate variables that pertain to Problem~\ref{prob:mean-pv-aug} only. Instead, we treat the dynamics \eqref{augsys} as being given, with $\mathbb{R}^n$-valued initial state $\eta_0$, and therefrom we will recover the unaugmented dynamics. Let $\mathbf{u} \in \mathcal{C}_3$ be given. Denote the first $n$ elements of $\eta_t$ as $x_t$ for each $t$; the dynamics of the reduced state are \emph{not} assumed to be \eqref{linsys}. Apply the definition of $f_t$ to \emph{derive} $x_{t+1} = A x_t + B u_t + w_t$ and $\eta_{t+1} = (x_{t+1}, D_t \eta_t)$  for each $t < N$. Then, use the latter expression and induction to find that $\eta_t = L_t (x_t, \dots, x_0)$ for some matrix $L_t$ for each $t \leq N$. Therefore, $(\eta_0, u_0,\dots,\eta_t)$ is a measurable function of $(x_0, u_0,\dots,x_t)$ for each $t \leq N$. It follows that $\mathbf{u} \in \mathcal{C}_2$.
\qed

Hence, Problems 2 and 3 are the same problem, just written in different forms. 
Overall, we have reformulated the original nonstandard problem (Problem~\ref{prob:mean-pv}) into a problem (Problem~\ref{prob:mean-pv-aug}) that enjoys a standard dynamic-programming-based solution.

\begin{corollary} \label{thm:main}
    Let all conditions and definitions pertaining to Problem~\ref{prob:mean-pv-aug} hold. Define the functions $J_t$ on $\mathbb{R}^{n_t}$  
    by $J_N \coloneqq c_N$, and for each $t = N-1,\dots,0$,
    \begin{equation} \label{eq:Jt}
        J_t(\eta) \coloneqq \inf_{u \in \mathbb{R}^m} \Big( c_t(\eta, u) + \mathbb{E}(J_{t+1}(f_t(\eta,u,w_t))) \Big).
    \end{equation}
    Then, $J_t$ is convex quadratic for every $ t \in \{0,\ldots,N\}$ and an affine function $\mu_t^*: \mathbb{R}^{n_t} \to \mathbb{R}^m$ attains the infimum in \eqref{eq:Jt} for every $ t \in \{0,\ldots,N-1\}$. In addition, $\mathbf{u}^* \coloneqq (u_0^*,\ldots,u_{N-1}^*)$, where $u_t^* \coloneqq \mu_t^* \circ \eta_t$, is an optimal controller for Problem~\ref{prob:mean-pv-aug}, 
    and the optimal value of Problem~\ref{prob:mean-pv-aug} is $\mathcal{J}_\lambda^*(\mathbf{x}_0) = J_0(\mathbf{x}_0)$.  
    Particularly, for every $t \in \{0,\ldots,N-1\}$ and $\eta \in \mathbb{R}^{n_t}$, we have
    \begin{equation*} 
        J_t(\eta)  = \eta^\top P_t \eta + q_t^\top \eta + r_{t} \quad \text{and} \quad   
        \mu_t^*(\eta)  = K_t \eta + \kappa_t, 
    \end{equation*}
    where $P_t \in \mathcal{S}_{n_t}$, $q_t$, $r_t$, $K_t$, and $\kappa_t$ satisfy the following Riccati recursion:
    \begin{equation*}\label{eq:recursion}\begin{aligned}
        P_t & \coloneqq - \tilde{A}_t^\top P_{t+1} \tilde{B}_t (\tilde{B}_t^\top P_{t+1} \tilde{B}_t + R)^{-1} \tilde{B}_t^\top P_{t+1} \tilde{A}_t \nonumber \\
        & \; \quad + \mathcal{Q}_{\lambda,t} + \tilde{A}_t^\top P_{t+1} \tilde{A}_t \\
        q_t & \coloneqq \lambda \zeta_t + (\tilde{A}_t + \tilde{B}_t K_t)^\top (q_{t+1} + 2P_{t+1} \tilde{C}_t \bar{w}) \\
        r_t & \coloneqq r_{t+1} + \mathrm{tr}((\Sigma + \bar{w}\bar{w}^\top) \tilde{C}_t^\top P_{t+1} \tilde{C}_t) + q_{t+1}^\top \tilde{C}_t \bar{w} \nonumber \\
        & \; \quad - \kappa_t^\top (\tilde{B}_t^\top P_{t+1} \tilde{B}_t + R) \kappa_t  \\
        K_t & \coloneqq -(\tilde{B}_t^\top P_{t+1} \tilde{B}_t + R)^{-1} \tilde{B}_t^\top P_{t+1} \tilde{A}_t  \\
        \kappa_t & \coloneqq \textstyle -\frac{1}{2}(\tilde{B}_t^\top P_{t+1} \tilde{B}_t + R)^{-1} \tilde{B}_t^\top (q_{t+1} + 2P_{t+1} \tilde{C}_t \bar{w})
    \end{aligned}\end{equation*}
    with $P_N \coloneqq \mathcal{Q}_{\lambda,N}$, $q_N \coloneqq \lambda \zeta_N$, and $r_N \coloneqq 0$.
\end{corollary}

\emph{Proof.}
The result follows from dynamic programming, e.g., see \cite[Th.~3.2.1]{hernandez2012discrete} and \cite[Ch.~6, Th.~2.15]{kumar2015stochastic}, 
where the model is nonstationary \cite[p. 32]{hernandez2012discrete} 
because the state space $\mathbb{R}^{n_t}$, the dynamics map $f_t$, and the stage cost function $c_t$ depend on time. The optimal values of Problems~\ref{prob:mean-pv-quad} and \ref{prob:mean-pv-aug} are equal (proof of Theorem~\ref{prop:equiv}); the optimal value of Problem~\ref{prob:mean-pv-aug} is $J_0(\mathbf{x}_0)$; and we denoted the optimal value of Problem~\ref{prob:mean-pv-quad} by $\mathcal{J}_\lambda^*(\mathbf{x}_0)$. Hence, $\mathcal{J}_\lambda^*(\mathbf{x}_0) = J_0(\mathbf{x}_0)$. \qed

Following standard DP, $u_t^*$ is a feedback depending on $\eta_t$ rather than the full history $g_t$. We note the following special cases. If $\mathbf{k}=0$, then $\mathbf{u}^*$
simplifies to the predictive variance controllers (without temporal coupling) in \cite[Th.~2]{tsiamis2020risk} and \cite[Th.~5]{tsiamis2025linear}. 
If furthermore $\lambda=0$ and $\bar{w}=0_{n\times 1}$, then $\mathbf{u}^*$ becomes the linear state-feedback controller for standard stochastic LQR ($\kappa_t$ becomes $0_{m \times 1}$).
While enjoying a familiar analytical form, $\mathbf{u}^*$ in the case of $\mathbf{k} \geq 1$ allows one to consider simultaneously the aims of state regulation and temporal variability reduction.

\section{NUMERICAL EXAMPLES} \label{sec:num}

Consider a point mass moving in a two-dimensional $\mu$--$\nu$-plane according to a discrete-time dynamics model $x_{t+1} = A x_t + Bu_t + B\xi_t$ for $t\in\{0,\ldots,N-1\}$, where
\begin{equation*}
    A = \begin{bmatrix}
        1 & T_s & 0 & 0\\
        0 & 1 & 0 & 0\\
        0 & 0 & 1 & T_s\\
        0 & 0 & 0 & 1
    \end{bmatrix}, \quad  B =  \begin{bmatrix}
        0 & 0\\
        \frac{T_s}{M} & 0\\
        0 & 0\\
        0 & \frac{T_s}{M}
    \end{bmatrix},
\end{equation*}
$T_s = 0.2$ is the sampling period, $M = 1$ is the mass, and $N = 100$.\footnote{Our code is available at \url{github.com/AlfredWei1/Temporal_Coupling_2026}}.
$x_t$ contains position $p_t \coloneqq (p_{t,\mu}, p_{t,\nu}) = (x_{t,1} , x_{t,3})$ and velocity; $x_{t,2}$ represents velocity in the $\mu$-direction. $u_t$ is the control force, while $\xi_t \coloneqq (\xi_{t,\mu} , \xi_{t,\nu})$ is the disturbance force. As well as $\xi_0,\dots,\xi_{N-1}$ being i.i.d., $\xi_{t,\mu}$ and $\xi_{t,\nu}$ are independent for each $t$.
The point mass occasionally experiences a jolt in the $\mu$-direction,
resembling examples from \cite{tsiamis2020risk, tsiamis2025linear, chapman2024risk}, modeled via the distribution of $\xi_t$.
Precisely, $\xi_{t,\mu}$ has a Gaussian mixture distribution $0.8 \times \mathcal{N}(0,10) + 0.2 \times \mathcal{N}(70,70)$, whereas $\xi_{t,\nu}$ has a Gaussian distribution $\mathcal{N}(0,10)$.
Hence, $B\xi_t$ satisfies Assumption~(A2).  
For a fixed $\mathbf{k}$, $\beta \geq 0$, and $Q \in \mathcal{S}_n$, we use an instance of $z_t$ \eqref{eq:zt} such that $z_t  = x_t^\top Qx_t + \beta \textstyle \sum_{i=1}^\mathbf{k} (x_t - x_{t-i})^\top Q (x_t - x_{t-i})$ for $t \geq \mathbf{k}$. 
We set $\mathbf{x}_0 = (5 , 0 , 5 , 0)$,  $R = I_2$, and $Q = \text{diag}(2,0.1,1,0.1)$. This $z_t$ resembles a weighted objective involving quadratic time differences, which has been used to find a smooth approximation to a corrupted signal \cite[p.~312]{vandenberghe2004convex}.
Applying this intuition to our setup, we steer the state toward the origin while smoothing its trajectory and penalizing stochastic variability.

In light of these three aims, we examine the effect of different choices of $\theta \coloneqq (\beta, \mathbf{k}, \lambda)$ on system performance using metrics that intuitively describe the behavior of the point mass, instead of ones that precisely match the objective~\eqref{eq:obj}:

\begin{itemize}
    \item Total distance traveled $\mathcal{D}_\text{tot} \coloneqq \sum_{t=1}^{N} |p_t - p_{t-1}|$;
    \item Total control effort $\mathcal{U}_{\text{tot}} \coloneqq \sum_{t=0}^{N-1} |u_t|$;
    \item Max distance from the origin $\mathcal{P}_\text{max} \coloneqq \max_{t=0,\dots,N} |p_t|$;
    \item Their respective means $\bar{\mathcal{D}}_\text{tot}$, $\bar{\mathcal{U}}_{\text{tot}}$, and  $\bar{\mathcal{P}}_\text{max}$ over some set of simulations;
    \item Length $\ell_{t,i}$ of the interval centered at the median of $x_{t,i}$ that contains 95\% of the samples of $x_{t,i}$ for $i \in \{ 1,2\}$.
\end{itemize}
We quantify temporal variability, control effort, stabilization, and stochastic variability in terms of  $\bar{\mathcal{D}}_\text{tot}$, $\bar{\mathcal{U}}_{\text{tot}}$, $\bar{\mathcal{P}}_\text{max}$, and $(\ell_{t,1},\ell_{t,2})$ vs. $t$, respectively.
Their relative importance in tuning $\theta$ is application-specific, and we leave the choice of an ``optimal'' $\theta$ to the user.
We use the notation $\theta = (\beta, \mathbf{k}, \lambda)$ to refer to a controller $\mathbf{u}^*$ with that choice of $\theta$.

First, we depict the behavior of the point mass when $\beta = \mathbf{k} = 1$ for three values of $\lambda$ in Fig. 1(a). As $\lambda$ increases, 
we observe reductions in stochastic and temporal variabilities, along with a rise in control effort. The tradeoff between stochastic variability and control effort is consistent with knowledge about the original predictive variance controller \cite[Figs. 3--4]{tsiamis2020risk}.
The case of $\lambda = 1$ ($\theta_2$) offers a middle ground with reduced stochastic variability and some deviation from the origin (see the mean positions near the origin). In contrast, the case of $\lambda = 6$ ($\theta_3$) exhibits more reduction in $\ell_{t,2}$ but also more deviation from the origin.

In Fig. 1(b), we consider the controller from \cite{tsiamis2020risk} ($\theta_4$) and two controllers with temporal coupling ($\theta_5$, $\theta_6$), all with $\lambda = 1$. At the cost of larger control effort, $\theta_5$ and $\theta_6$ lead to smaller temporal variability ($\bar{\mathcal{D}}_\text{tot}$); graphically, their mean positions in the early horizon are more densely spaced. Thus, temporal coupling enables the generation of slowly varying trajectories, as designed. We also find that temporal coupling can reduce uncertainty. Despite $\lambda$ being fixed, stochastic variability is typically smaller in the cases of $\theta_5$ and $\theta_6$ compared to $\theta_4$. Seeing a large increase in $\bar{\mathcal{U}}_\text{tot}$ but little change in $\bar{\mathcal{D}}_\text{tot}$ in the case of $\theta_6$ versus that of $\theta_5$, next we study the effect of varying $\beta$, $\mathbf{k}$, and $\lambda$ simultaneously. 

We present $\bar{\mathcal{D}}_{\text{tot}}$ and $\bar{\mathcal{U}}_{\text{tot}}$ versus $\bar{\mathcal{P}}_{\text{max}}$ in Fig. 2 for performance analysis.
In Fig. 2(i.c) where $\mathbf{k} = \lambda = 1$, as $\beta$ increases, $\bar{\mathcal{D}}_{\text{tot}}$ only reduces slightly whereas $\bar{\mathcal{U}}_{\text{tot}}$ quickly increases drastically.
However, for some choices of $\mathbf{k} > 1$ and $\lambda > 0$, the $\bar{\mathcal{U}}_{\text{tot}}$ versus $\bar{\mathcal{P}}_{\text{max}}$
relationship is not monotonic (Fig. 2(ii.b, iii.b, ii.c, iii.c)), and the parameter choices corresponding to the upper arc of $\bar{\mathcal{U}}_{\text{tot}}$-$\bar{\mathcal{P}}_{\text{max}}$ curve can be less preferable in terms of these two metrics. 
In the case of $\mathbf{k} = 9$ and $\lambda = 0$ (Fig. 2(iii.a)), the $\bar{\mathcal{D}}_\text{tot}$ versus $\bar{\mathcal{P}}_{\text{max}}$ curve continues to improve as $\beta$ increases, but with rising control effort; $\beta = 10$ achieves the lowest values of $\bar{\mathcal{D}}_\text{tot}$ and $\bar{\mathcal{P}}_{\text{max}}$ simultaneously across all parameter choices in Fig. 2. Yet, this benefit comes at the cost of larger control effort compared to the choices we investigate next. 

We consider $\theta_8 \coloneqq (2,9,0)$  and $\theta_9 \coloneqq (1.5,9,0.2)$, which have approximately the same $\bar{\mathcal{U}}_{\text{tot}}$. 
Fig. 1(b,c) presents the behavior of the point mass under $\theta_8$, $\theta_9$, and the benchmarks $\theta_4$, $\theta_7 \coloneqq (0,0,0)$.
Compared to $\theta_7$, $\theta_4$ reduces stochastic variability (Fig. 1(c)) but exerts more control effort, a known tradeoff \cite{tsiamis2020risk} (mentioned before). $\theta_4$ also reduces temporal variability compared to $\theta_7$.
Exerting even more control effort, $\theta_8$ and $\theta_9$ further reduce stochastic and temporal variabilities.
$\theta_8$ outperforms $\theta_9$ in terms of $\ell_{t,1}$ after about $t = 30$, whereas $\theta_9$ generally outperforms $\theta_8$ in terms of $\ell_{t,2}$. The ability of $\theta_9$ to reduce uncertainty in $x_{t,2}$ is not so surprising; most entries of $\Sigma$ and $\gamma$ are zero or small, but $\Sigma_{22}$ and $\gamma_2$, which correspond to $x_{t,2}$, are large. The superior uncertainty reduction in $x_{t,2}$ under $\theta_9$ comes at the cost of more deviation from the origin; graphically, in the late horizon, $\theta_9$ produces a larger drift in the mean positions than $\theta_8$ does.

Summarizing, we illustrated various tradeoffs brought by altering $\theta$, and the controlled system behavior is rich despite the dynamics being rudimentary for pedagogical purpose.
We experimented with parameter values and used plots of $\bar{\mathcal{D}}_\text{tot}$, $\bar{\mathcal{U}}_\text{tot}$, and $\bar{\mathcal{P}}_\text{max}$ (e.g., Fig. 2) to select values that led to interpretable trends. The results are specific to the setup and parameter values. For instance, in a preliminary experiment, a different choice of $(\beta,\mathbf{k})$ than that in Fig. 1(a) did not consistently reduce $\ell_{t,1}$ with increasing $\lambda$---a reasonable finding because $\mathbf{u}^*$ (with $\lambda>0$) is designed to penalize uncertainty in $z_t$~\eqref{eq:zt}, which involves more states than $x_{t,1}$.
For a different case study, we suggest selecting parameter values based on their simulated performance,
where an acceptable range of values may be determined first by plotting performance tradeoffs (e.g., Fig. 2).

 \begin{figure*}[htb]
\captionsetup{justification=raggedright,singlelinecheck=false}
 \begin{subfigure}[t]{\textwidth}
            \centering
            \includegraphics[width=\textwidth]{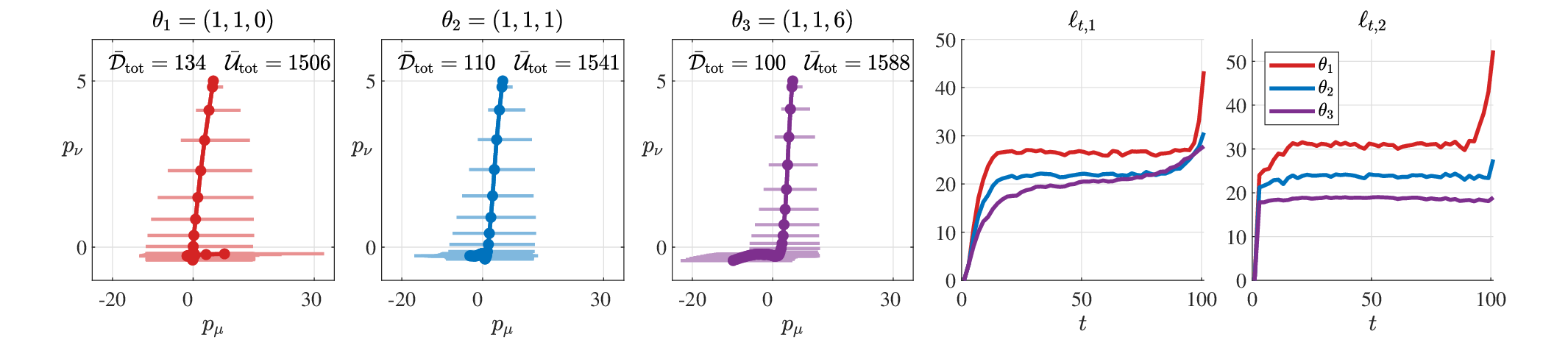}\vspace{-2mm}
            \caption{\normalsize First experiment for controllers $\theta_1$, $\theta_2$, $\theta_3$, with $\mathbf{\beta} = \mathbf{k} = 1$ and $\lambda\in\{0,1,6\}$.} \vspace{2mm}
            \label{fig:fig1a}
    \end{subfigure}

    \begin{subfigure}[t]{\textwidth}
            \centering
            \includegraphics[width=\textwidth]{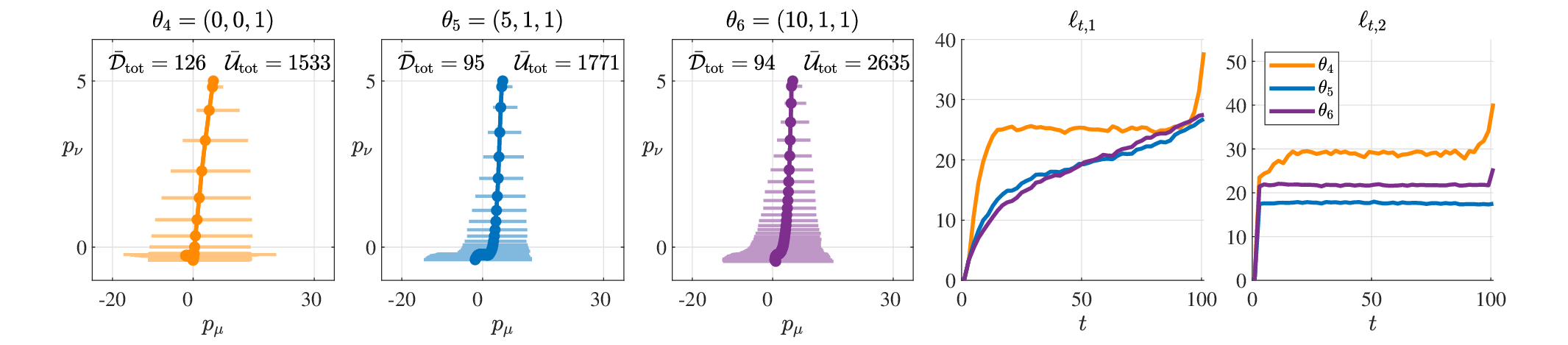}\vspace{-2mm}
            \caption{\normalsize Second experiment for controllers $\theta_4$, $\theta_5$, $\theta_6$, with $\lambda = 1$; $\theta_4$ is from \cite{tsiamis2020risk}; $(\beta,\mathbf{k}) = (5,1)$ for $\theta_5$ and $(10,1)$ for $\theta_6$.} \vspace{2mm}
            \label{fig:fig1b}
    \end{subfigure}

    \begin{subfigure}[t]{\textwidth}
            \centering
    \includegraphics[width=\textwidth]{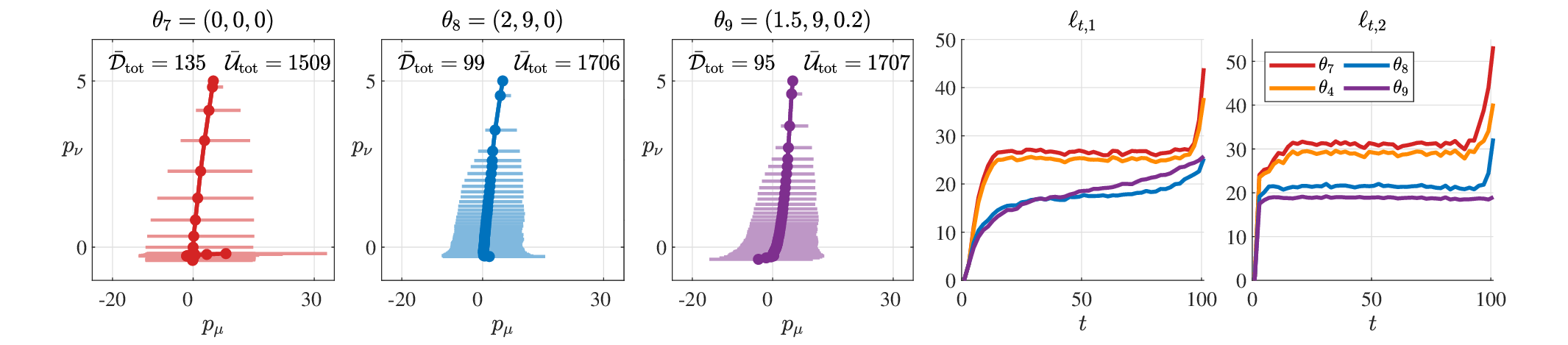}\vspace{-2mm}
            \caption{\normalsize Third experiment for controllers $\theta_7$, $\theta_8$, $\theta_9$; we repeat $\theta_4$ from (b) in the $\ell_{t,1}$ and $\ell_{t,2}$ plots for the sake of comparison.} \vspace{2mm}
            \label{fig:fig1c}
    \end{subfigure}
    \caption{Performance of various controllers. Each row compares a set of controllers. Columns 1--3: Trajectories of the point mass from 5000 simulations under $\mathbf{u}^*$ with some choice of $\theta = (\beta,\mathbf{k},\lambda)$, and the corresponding values of $\bar{\mathcal{D}}_\text{tot}$ and $\bar{\mathcal{U}}_\text{tot}$; for each $t \in \{0,2,4,\dots,100\}$, we show the associated empirical mean position (dot) and interval of length $\ell_{t,1}$ containing 95\% of the samples of $p_{t,\mu}$ (centered at the median of $p_{t,\mu}$). Columns 4--5: $\ell_{t,1}$ and $\ell_{t,2}$ versus $t$.}
     \label{fig:main_fig}
 \end{figure*}

\begin{figure}[h]
\definecolor{Dtotal}{RGB}{0, 114, 189}
\definecolor{Utotal}{RGB}{217, 83, 25}
\centering
\includegraphics[width = \columnwidth]{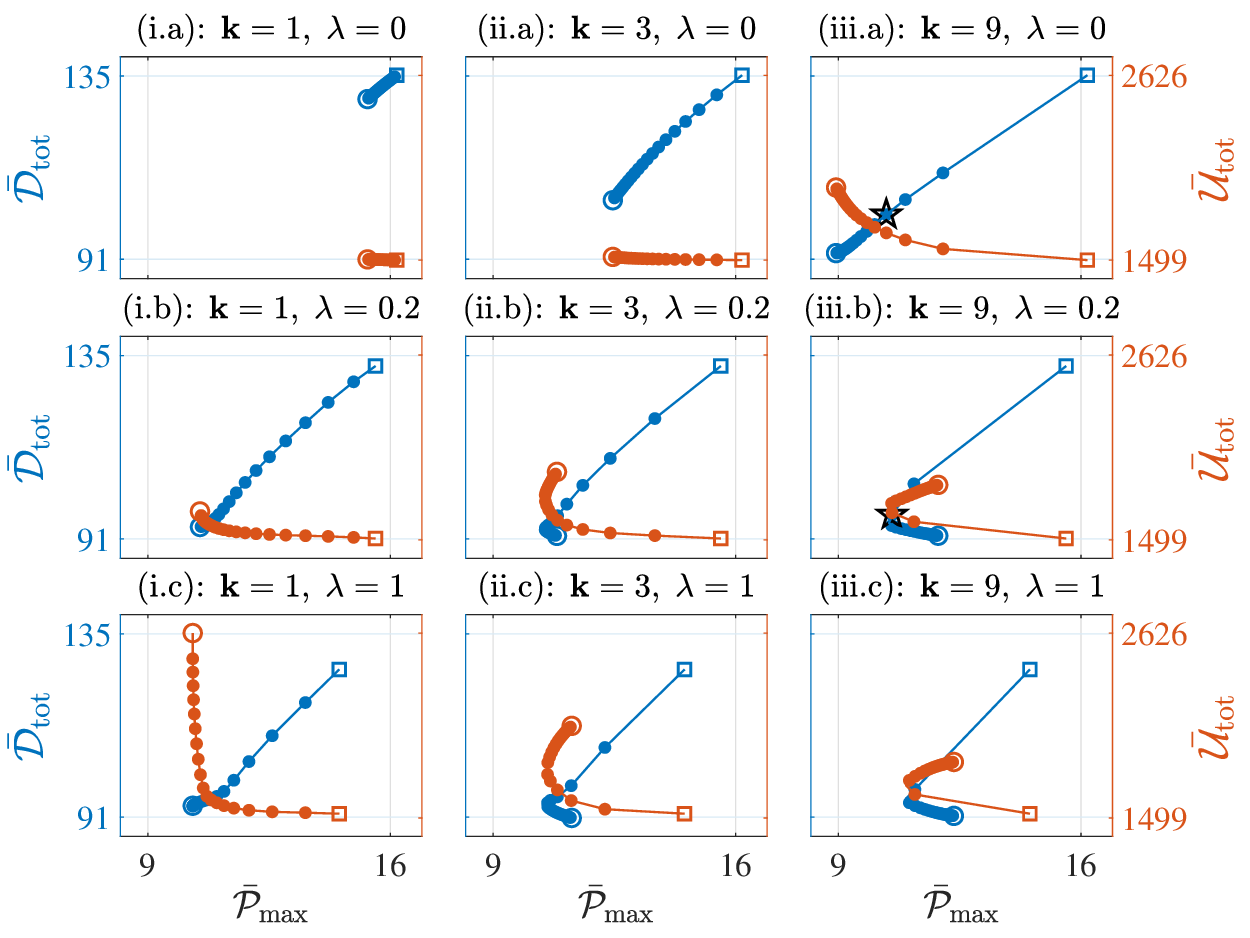}
\caption{$\bar{\mathcal{D}}_\text{tot}$ versus $\bar{\mathcal{P}}_\text{max}$ (blue) and  $\bar{\mathcal{U}}_\text{tot}$ versus $\bar{\mathcal{P}}_\text{max}$ (red) calculated from $500$ simulations for each $\theta = (\beta,\mathbf{k},\lambda)$. In each subplot, $\beta \in \{0,0.5,\dots,10\}$ is varied for some fixed $(\mathbf{k},\lambda)$. The square and circle markers indicate $\beta = 0$ and $\beta = 10$, respectively. Note that $\beta=0$ represents the absence of temporal coupling and corresponds to the controller from \cite{tsiamis2020risk} if $\lambda > 0$. A star marker corresponds to $\theta_8 = (2,9,0)$ or $\theta_9=(1.5,9,0.2)$.} \label{fig:comparison}
\end{figure}

\section{CONCLUSION} \label{sec:conclusion}
We develop a controller that shapes a stochastic state trajectory, penalizing its variability over time and also in a probabilistic sense.
The simulations demonstrate that penalizing period-wise behavior can reduce trajectory variability on average. We also present a parameter evaluation (Fig. \ref{fig:comparison}) that illustrates the effects of different parameter choices with respect to three metrics. The form of $z_t$ with which we couple past states can be discussed in relation to separability notions in dynamic programming problems \cite{li1991extension,jones2020extensions}. Theoretical extensions to nonlinear, partially observable systems may facilitate real-world applications that call for mitigating temporal and stochastic variabilities.

\section*{APPENDIX}
To prove Lemma~\ref{lemma:delta-t-decomposition}, 
first we prove the decomposition $\Delta_t^2 = \Delta_{t,1} + \Delta_{t,2} + \Delta_{t,3}$ a.s. Then, 
we consider the individual terms in the sum. $\Delta_{t,1}$ was found to be integrable and its expectation was computed using elementary probability theory in \cite[Appendix]{tsiamis2020risk}. 
Drawing from \cite[Appendix]{tsiamis2020risk}, we justify why $\Delta_{t,2}$ and $\Delta_{t,3}$ are integrable and compute their expectations 
using elementary probability  theory \cite{ash1972probability}
and any past state $x_\tau$  
for $\tau \in \{0,\dots,t-1\}$
being $\mathcal{F}_{t-1}$-measurable. Fact 1 below can be proved using \cite{ash1972probability}.

\begin{fact} \label{prop:expt-vector}
    Let $x$ and $y$ be $\mathbb{R}^n$-valued random vectors defined on  
    $(\Omega, \mathcal{G}, \mathbb{P})$, whose expectations exist. 
    \begin{enumerate}
        \item[(a)] Let $\mathcal{F}$ be a sub-$\sigma$-algebra of $\mathcal{G}$. If $y$ is $\mathcal{F}$-measurable and $\mathbb{E}(|x_i|), \mathbb{E}(|x_iy_i|) < \infty$ for every $i \in \{1, \ldots, n\}$, then $\mathbb{E}(x^\top y|\mathcal{F}) = \mathbb{E}(x|\mathcal{F})^\top y = y^\top \mathbb{E}(x|\mathcal{F})$ a.s.
        \item[(b)] If $x$ and $y$ are independent and integrable, then $x^\top y$ is integrable and $\mathbb{E}(x^\top y) = \mathbb{E}(x)^\top \mathbb{E}(y)$.
    \end{enumerate}
\end{fact}

We now proceed to proving Lemma~\ref{lemma:delta-t-decomposition}. 

\textit{Step 1:} First, consider the case of $\mathbf{k}_t \geq 1$. Define
$ \hat{x}_t \coloneqq \mathbb{E}(x_t|\mathcal{F}_{t-1}) = Ax_{t-1} + Bu_{t-1} + \bar{w} $ a.s. \cite{tsiamis2020risk},
which is $\mathcal{F}_{t-1}$-measurable. Since $x_0,\dots,x_N$ are square-integrable random vectors under Assumption~\ref{assu1}, all terms in the expansion~\eqref{eq:zt-expanded} of $z_t$ are integrable random variables. Then, 
\begin{equation*}\begin{aligned}
    \mathbb{E}(z_t|\mathcal{F}_{t-1}) \overset{\text{a.s.}}{=} & \textstyle \sum_{i=1}^{\mathbf{k}_t} \sum_{j=1}^{\mathbf{k}_t} x_{t-i}^\top Q_{ij} x_{t-j} \\
    & \textstyle + 2 \sum_{i=1}^{\mathbf{k}_t} x_{t-i}^\top Q_{i0} \hat{x}_t + \mathbb{E}(x_t^\top Q_{00} x_t | \mathcal{F}_{t-1}) 
\end{aligned}\end{equation*}
in particular by linearity of conditional expectation \cite[Th. 6.5.2(a)]{ash1972probability} 
and Fact~\ref{prop:expt-vector}(a) with $x_{t-i}$ being $\mathcal{F}_{t-1}$-measurable for every $i \in \{1, \ldots, \mathbf{k}_t\}$. 
Using $x_t - \hat{x}_t = d_t$ a.s. \cite{tsiamis2020risk}, we can compute $\Delta_t$ from its definition to find that 
$\Delta_t = 2 \sum_{i=1}^{\mathbf{k}_t} x_{t-i}^\top Q_{i0} d_t + y_t$ \text{a.s.} 
Squaring this quantity gives the decomposition for $\Delta_t^2$. Now, in the case of $\mathbf{k}_t = 0$, we have $z_t = x_t^\top Q_{00} x_t$, leading to $\Delta_t = y_t$.

\textit{Step 2a--Computation of $\mathbb{E}(\Delta_{t,2})$:} For each $i$ and $j$ in $\{1,\ldots,\mathbf{k}_t\}$, the random vectors $x_{t-i}$ and $x_{t-j}$ are $\mathcal{F}_{t-1}$-measurable and square-integrable, and the random matrix $Q_{i0} d_t d_t^\top Q_{0j}$ is integrable and independent of $h_{t-1}$. Hence, $x_{t-i}^\top Q_{i0} d_t d_t^\top Q_{0j} x_{t-j}$ is integrable and $\mathbb{E}(x_{t-i}^\top Q_{i0} d_t d_t^\top Q_{0j} x_{t-j}|\mathcal{F}_{t-1}) = x_{t-i}^\top Q_{i0} \Sigma Q_{0j} x_{t-j}$ a.s. \cite[Lemma 4.2]{patel2025risk}. Taking the expectation, applying the law of total expectation, and summing over $i$ and $j$ leads to \eqref{term2-expt}. 

\textit{Step 2b--Computation of $\mathbb{E}(\Delta_{t,3})$:} It can be shown that
\begin{equation} \label{eq:yt} 
   y_t = d_t^\top Q_{00} d_t - \mathrm{tr}(\Sigma Q_{00}) + 2\hat{x}_t^\top Q_{00} d_t \quad \text{a.s.}
\end{equation}
\cite[Appendix]{tsiamis2020risk}. We substitute \eqref{eq:yt} into \eqref{term3} to obtain
    $ \Delta_{t,3} = 4 \sum_{i=1}^{\mathbf{k}_t} \alpha_{1i} + \alpha_{2i} + \alpha_{3i} $ a.s.,
where 
\begin{equation*}\begin{aligned}
    \alpha_{1i} & \coloneqq x_{t-i}^\top Q_{i0} d_t d_t^\top Q_{00} d_t, \\
    \alpha_{2i} & \coloneqq -x_{t-i}^\top Q_{i0} d_t \text{tr}(\Sigma Q_{00}), \quad \text{and}\\
    \alpha_{3i} & \coloneqq 2 x_{t-i}^\top Q_{i0} d_t d_t^\top Q_{00} \hat{x}_t.
\end{aligned}\end{equation*}
Every $\alpha_{1i}$ and $\alpha_{2i}$ are integrable by Fact \ref{prop:expt-vector}(b), and every $\alpha_{3i}$ is integrable \cite[Lemma 4.2]{patel2025risk}, so $\Delta_{t,3}$ is integrable. Computing their expectations, $\mathbb{E}(\alpha_{1i}) = \mathbb{E}(x_{t-i}^\top Q_{i0} \gamma )$, $\mathbb{E}(\alpha_{2i}) = 0$, and $\mathbb{E}(\alpha_{3i}) = 2\mathbb{E}(x_{t-i}^\top Q_{i0} \Sigma Q_{00} x_t)$,
leads to \eqref{term3-expt}. \qed

\section*{ACKNOWLEDGMENT}   
We thank Manfredi Maggiore, Laurent Lessard, Marco Pavone, Anirudha Majumdar, and Dhairya Patel for advice. 

\bibliographystyle{IEEEtran}         
\bibliography{autosam}          
 
\end{document}